\documentclass[runningheads, envcountsame]{llncs}
\usepackage{amsmath,amssymb} % AMS symbol library

\usepackage{verbatim}		% for {comment} environment

\usepackage{tikz}			% tikz package for pictures
\usetikzlibrary{arrows}

\usepackage{url}			% for \url command

\usepackage{lipsum}

\newcommand{\id}{\mathrm{id}}
\newcommand{\U}{\mathcal{U}}

\newcommand{\lincol}{black}
\newcommand{\linth}{thick}
\newcommand{\po}[2][\pocol]{\filldraw[#1](#2) circle (2 pt);}
\newcommand{\li}[1]{\draw[\linth,\lincol] #1;}
\begin{document}
\mainmatter              		% start of the contribution
\pagestyle{headings}		% for page numbers 
\title{Duality and universal models for the meet-implication fragment of IPC}
\titlerunning{On duality and universal models}  % abbreviated title (for running head)
\author{Nick Bezhanishvili\inst{1} \and Dion Coumans\inst{2} \and \\ Sam van Gool\inst{2,3} \and Dick de Jongh\inst{1}}
\authorrunning{N. Bezhanishvili, D. C. S. Coumans, S. J. van Gool, D. H. J. de Jongh} % abbreviated author list (for running head)
\tocauthor{N. Bezhanishvili, D. C. S. Coumans, S. J. van Gool and D. H. J. de Jongh}

\institute{ILLC, Faculty of Science, University of Amsterdam \\
Science Park 107, 1098 XG Amsterdam, The Netherlands
\and
IMAPP, Faculty of Science, Radboud University Nijmegen \\
P. O. Box 9010, 6500 GL Nijmegen, The Netherlands
\and
Mathematical Institute, University of Bern\\
Sidlerstrasse 5, 3012 Bern, Switzerland}

\maketitle              % typeset the title of the contribution

\begin{abstract}
In this paper we investigate the
fragment of intuitionistic logic which only uses conjunction (meet) and implication, using finite duality for distributive lattices and universal models. 
We give a description of the finitely generated universal models of this fragment and 
give a complete characterization of the up-sets of Kripke models of intuitionistic logic which can be defined by meet-implication-formulas. We use these
results to derive a new version of subframe formulas for intuitionistic logic and 
to show that the uniform interpolants of meet-implication-formulas
are not necessarily uniform interpolants in the full intuitionistic logic.

\end{abstract}

\section{Introduction}
\label{sec:intro}

Heyting algebras are the algebraic models of intuitionistic propositional logic, IPC. In this paper we will be concerned with the syntactic fragment of IPC consisting of the formulas which only use the connectives of conjunction ($\wedge$) and implication ($\to$), but no disjunction ($\vee$) or falsum ($\bot$). The algebraic structures corresponding to this fragment are called implicative semilattices\footnote{In less recent literature, these are also called Brouwerian semilattices.}. A result due to Diego \cite{Diego} says that the variety of implicative semilattices is locally finite, i.e., finitely generated algebras are finite, or equivalently, the finitely generated free algebras are finite. In logic terms, this theorem can be expressed as saying that there are only finitely many equivalence classes of $(\wedge,\to)$-formulas in IPC. 

One of the key results in this paper is a dual characterization of a $(\wedge,\to)$-subalgebra  of a given Heyting algebra generated by a finite set of elements (Theorem~\ref{thm:gensubdual}).
This theorem leads to Diego's theorem and a characterization of the $n$-universal models of the $(\wedge,\to)$-fragment of IPC (Theorem~\ref{thm:diegostrong}) as submodels of the universal model, in the same spirit as the proof by Renardel de Lavalette et al. in \cite{RenHenJon}. The first characterization of this model was obtained by K\"ohler \cite{Kohler} using his duality for finite implicative meet-semilattices. Our slightly different approach in this paper also enables us to obtain new results about the $(\wedge,\to)$-fragment of IPC. In particular, in Theorem~\ref{thm:characterization}, we give a full characterization of the up-sets of a Kripke model which can be defined by $(\wedge,\to)$-formulas. Since our characterization in particular applies to the $n$-universal model of IPC, this may be considered as a first step towards solving the complicated problem of characterizing the up-sets of the 
$n$-universal models which are definable by intuitionistic formulas (also see our more detailed remarks in Section~\ref{sec:conclusion}). 
Building on this result, we use the de Jongh formulas for IPC to construct formulas that play an analogous role 
in the $(\wedge,\to)$-fragment.  
Finally, we use the characterization of $(\wedge,\to)$-definable subsets of the $n$-universal models of IPC to show that 
a uniform interpolant of a $(\wedge,\to)$-formula in intuitionistic logic may not be equivalent to a $(\wedge,\to)$-formula.\\

A word on methodology. The two essential ingredients to our proofs are, on the one hand, Birkhoff duality for finite distributive lattices and, on the other hand, the theory of $n$-universal models for IPC (\cite{deJ68}, \cite{NBezh}). Our methods in this paper are directly inspired by the theory of duality for $(\wedge, \to)$-homomorphisms as developed in \cite{Kohler}, \cite{ChaZak1997}, \cite{GBezhJans}, \cite{BezhBezhCan}, and also by the observations about the relation between the $n$-universal models and duality for Heyting algebras in \cite{Geh2013Esakia}. However, we made an effort to write this paper in such a way to be as self-contained as possible, and in particular we do not require the reader to be familiar with any of these results. In particular, we give a brief introduction to duality for finite distributive lattices and its connection to Kripke semantics for IPC in Section~\ref{sec:prelim}, and we do not need to go into the intricacies of duality for implicative meet-semilattices, instead opting to give direct proofs of the duality-theoretic facts that we need.\\

The paper is organized as follows: in Section~\ref{sec:prelim} we present the necessary preliminaries about IPC and Heyting algebras in the context of duality for distributive lattices; in Section~\ref{sec:sepmeetimp} we study the meet-implication fragment of IPC and prove our main theorems mentioned above; in Section~\ref{sec:subframeunifint} we apply these results to $(\wedge,\to)$-de Jongh formulas and analyze semantically the uniform interpolation in the $(\wedge,\to)$-fragment of IPC. In Section~\ref{sec:conclusion} we summarize our results and give suggestions on where to go from here.

\section{Algebra, semantics and duality}
\label{sec:prelim}
We briefly outline the contents of this section. In Subsection~\ref{subsec:adjunction}, we recall the definitions and basic facts about adjunctions between partially ordered sets, Heyting algebras, and implicative meet-semilattices. Subsection~\ref{subsec:duality} contains the preliminaries about duality theory that we will need in this paper. In Subsection~\ref{subsec:semanticsduality} we show how to define the usual Kripke semantics for IPC via duality, and in Subsection~\ref{subsec:canunivmod} we recall how the universal and canonical models for IPC are related to free finitely generated Heyting algebras via duality.
\subsection{Adjunction, Heyting algebras, implicative meet-semilattices}\label{subsec:adjunction}
Since the notion of adjunction is crucial to logic in general, and in particular to intuitionistic logic, we recall some basic facts about it right away. An adjunction can be understood as an invertible rule that ties two logical connectives or terms. The typical example in intuitionistic logic is the adjunction between $\wedge$ and $\to$, which can be expressed by saying that the following (invertible) rule is derivable in IPC.
\begin{equation}\label{eq:adjrule}
\begin{tikzpicture}[baseline=(current  bounding  box.center)]
\node at (0,.5) {$p \wedge q \vdash r$};
\draw (-1,.22) -- (1,.22);
\draw (-1,.28) -- (1,.28);
\node at (0,0) {$p \vdash q \to r$};
\end{tikzpicture}
\end{equation}
Recall that an {\it adjunction} between partially ordered sets $A$ and $B$ is a pair of functions $f : A \leftrightarrows B : g$ such that, for all $a \in A$ and $b \in B$, $f(a) \leq b$ if, and only if, $a \leq g(b)$; notation: $f \dashv g$. In this case, we say that $f$ is {\it lower adjoint} to $g$ and $g$ is {\it upper adjoint} to $f$. Note that the derivability of rule (\ref{eq:adjrule}) in IPC says exactly that, for any $\psi$, the function $\varphi \mapsto \varphi \wedge \psi$ on the Lindenbaum algebra for IPC (cf. Example~\ref{exa:heytingalgebras}(b) below) is lower adjoint to the function $\chi \mapsto \psi \to \chi$. The following general facts about adjunctions are well-known and will be used repeatedly in this paper.
\begin{proposition}\label{prop:adjbasicfacts}
Let $A$ and $B$ be partially ordered sets and let $f : A \leftrightarrows B : g$ be an adjunction. The following properties hold:
\begin{enumerate}
\item If $f$ is surjective, then $fg = \id_B$, and therefore $g$ is injective and the image of $g$ is $\{a \in A \ | \ gf(a) \leq a\}$;
\item The function $f$ preserves any joins (suprema) which exist in $A$ and the function $g$ preserves any meets (infima) which exist in $B$;
\item For any $b \in B$, $g(b)$ is the maximum of $\{a \in A \ | \ f(a) \leq b\}$. In particular, the fact that $g$ is upper adjoint to $f$ uniquely determines $g$.
\end{enumerate}
Moreover, if $C$ and $D$ are complete lattices and $f : C \to D$ is a function which preserves arbitrary joins, then $f$ has an upper adjoint.
\end{proposition}
\begin{proof}
Straightforward; cf., e.g., \cite[7.23--7.34]{DavPri2002}.\qed
\end{proof}

Recall that a tuple $(A, \wedge, \vee, \to, 0, 1)$ is a {\it Heyting algebra} if $(A, \wedge, \vee, 0, 1)$ is a bounded lattice, and the operation $\to$ is upper adjoint to $\wedge$, i.e., for any $a, b, c \in A$,
\begin{equation}
\label{eq:heyting}
a \wedge b \leq c \iff a \leq b \to c.
\end{equation}
The equation (\ref{eq:heyting}) says that $b \to c$ is the maximum of $\{a \in A \ | \ a \wedge b \leq c\}$; therefore, a lattice admits at most one ``Heyting implication'', i.e., an operation $\to$ such that it becomes a Heyting algebra. The lattices underlying Heyting algebras are always distributive (in fact, for any $a \in A$, the function $b \mapsto a\wedge b$ preserves any join that exists in $A$, since it is a lower adjoint). All finite distributive lattices admit a Heyting implication. A {\it Heyting homomorphism} is a map between Heyting algebras that preserves each of the operations.
An implicative meet-semilattice is a ``Heyting algebra without disjunction''. More precisely, an {\it implicative meet-semilattice} is a tuple $(A,\wedge,\to)$ such that $(A,\wedge)$ is a semilattice, and condition (\ref{eq:heyting}) holds. We will write $(\wedge,\to)$-homomorphism to abbreviate ``homomorphism of implicative meet-semilattices''. Note that any implicative meet-semilattice has a largest element, $1$, which is preserved by any $(\wedge,\to)$-homomorphism. Also note that finite implicative meet-semilattices are distributive lattices, but $(\wedge,\to)$-homomorphisms do not necessarily preserve joins. However, {\it surjective} $(\wedge,\to)$-homomorphisms do preserve join (cf. \cite[Lemma 2.4 and the remark thereafter]{Kohler}):
\begin{lemma}\label{lem:surjheyt}
If $f : A \to B$ is a surjective $(\wedge,\to)$-homomorphism between Heyting algebras, then $f$ is join-preserving.
\end{lemma}
\begin{proof}
First of all, we have $0_B = f(a)$ for some $a \in A$, and $0_A \leq a$, so that $f(0_A) = 0_B$. Now let $a, a' \in A$. Pick $c \in A$ such that $f(c) = f(a) \vee f(a')$. Now 
\begin{align*}
f(a \vee a') \to (f(a) \vee f(a')) &= f(a \vee a') \to f(c) \\
&= f((a \vee a') \to c) \\
&= f((a \to c) \wedge (a' \to c)) \\
&= (f(a) \to f(c)) \wedge (f(a') \to f(c)) \\
&= (f(a) \vee f(a')) \to f(c) = 1,
\end{align*}
so $f(a \vee a') \leq f(a) \vee f(a')$. The other inequality holds because $f$ is order-preserving.\qed
\end{proof}

\begin{example}\label{exa:heytingalgebras}
(a) An important example of a Heyting algebra is the collection of upward closed sets (`up-sets') of a partially ordered set $(X,\leq)$, ordered by inclusion; we denote this Heyting algebra by $\mathcal{U}(X)$. The Heyting implication of two up-sets $U$ and $V$ is given by the formula
\begin{equation}
\label{eq:heytingupsets}
U \rightarrow V = ({\downarrow}(U \cap V^c))^c,
\end{equation}
that is, a point $x$ is in $U \rightarrow V$ if, and only if, for all $y \geq x$, $y \in U$ implies $y \in V$. The reader who is familiar with models for IPC will recognize the similarity between this condition and the interpretation of a formula $\varphi \to \psi$ in a model; we will recall the precise connection between the two in \ref{subsec:semanticsduality} below.

(b) Another example of a Heyting algebra, of a more logical nature, is that of the Lindenbaum algebra for IPC; we briefly recall the definition. Fix a set of propositional variables $P$ and consider the collection $F(P)$ of all propositional formulas whose variables are in $P$. Define a pre-order $\preceq$ on $F(P)$ by saying, for $\varphi, \psi \in F(P)$, that $\varphi \preceq \psi$ if, and only if, $\psi$ is provable from $\varphi$ in IPC. The Lindenbaum algebra is defined as the quotient of $F(P)$ by the congruence relation \break ${\approx} := (\preceq) \cap (\preceq)^{-1}$. The Lindenbaum algebra is the free Heyting algebra over the set $P$, i.e., any function from $P$ to a Heyting algebra $H$ lifts uniquely to a Heyting homomorphism from the Lindenbaum algebra over $P$ to $H$. We will denote the free Heyting algebra over $P$ by $F_{HA}(P)$. Note that the same construction can be applied to the $(\wedge,\vee)$- and $(\wedge,\to)$-fragments of IPC to yield the free distributive lattice $F_{DL}(P)$ and the free implicative meet-semilattice $F_{\wedge,\to}(P)$, respectively. \qed
\end{example}

\subsection{Duality}
\label{subsec:duality}
We briefly recall the facts about duality that we will need. Let $D$ be a distributive lattice. We recall the definition of the {\it dual poset}, $D_*$, of $D$. The points of $D_*$ are the prime filters of $D$, i.e., up-sets $F \subseteq D$ which contain finite meets of their subsets and have the property that if $a \vee b \in F$, then $a \in F$ or $b \in F$. The partial order on $D_*$ is the inclusion of prime filters. The map $\eta : D \to \U(D_*)$ which sends $d \in D$ to $\{F \in D_* \ | \ d \in F\}$ is (assuming the axiom of choice) an embedding of distributive lattices, which is called the {\it canonical extension} of $D$. If $D$ is finite, then $\eta$ is an isomorphism, so that any finite distributive lattice is isomorphic to the lattice of up-sets of its dual poset. The assignments $X \mapsto \U(X)$ and $D \mapsto D_*$ between finite posets and finite distributive lattices extend to a dual equivalence, or duality, of categories: homomorphisms from a distributive lattice $D$ to a distributive lattice $E$ are in a natural bijective correspondence with order-preserving maps from $E_*$ to $D_*$. A homomorphism $h : D \to E$ is sent to the map $h_* : E_* \to D_*$ which sends $F \in E_*$ to $h^{-1}(F)$, and an order-preserving map $f : X \to Y$ is sent to the homomorphism $f^* : \U(Y) \to \U(X)$ which sends an up-set $U$ of $Y$ to $f^{-1}(U)$.

If $X$ and $Y$ are posets, it is natural to ask which order-preserving maps $f : X \to Y$ are such that their dual $f^{-1} : \U(Y) \to \U(X)$ is a Heyting homomorphism. It turns out that these are the {\it p-morphisms}, i.e., the order-preserving maps which in addition satisfy the condition: for any $x \in X$, $y \in Y$, if $f(x) \leq y$, then there exists $x' \geq x$ such that $f(x') = y$. 

To end this subsection, we recall how duality yields a straight-forward description of the free finitely generated\footnote{Essentially the same argument as the one sketched in this paragraph can be used to give a description of an arbitrary, not necessarily finitely generated, free distributive lattice, but we will not need this in what follows.} distributive lattice, $F_{DL}(P)$. In any category of algebras, the free algebra over a set $P$ is the $P$-fold coproduct of the one-generated free algebra. Therefore, since duality transforms coproducts into products, the dual space $F_{DL}(P)_*$ is the $P$-fold power of the poset $F_{DL}(\{p\})_*$, the dual of the one-generated free algebra. Note that $F_{DL}(\{p\})$ is the three-element chain $\{0 \leq p \leq 1\}$, so its dual is the two-element poset $2 = \{0,1\}$. Since finite products in the category of finite posets are simply given by equipping the Cartesian product with the pointwise order, it follows that $F_{DL}(P)_* = 2^P$. Therefore, the free distributive lattice over a finite set $P$ is the lattice of up-sets of $2^P$; in a formula, $F_{DL}(P) = \U(2^P)$.
\subsection{Semantics via duality}
\label{subsec:semanticsduality}
{\bf Notation.} Throughout the rest of this paper, we fix a finite set of propositional variables $P = \{p_1,\dots,p_n\}$. We denote the free algebras over $P$ by $F_{HA}(n)$, $F_{DL}(n)$, etc.

In this paper, a {\it frame} is a poset $(M,\leq)$. A {\it model} is a triple $(M,\leq,c)$, where $(M,\leq)$ is a poset and $c$, the {\it colouring}, is an order-preserving function from $M$ to $2^n$. 
The colouring $c$ yields, via duality, a distributive lattice homomorphism $c^* : \U(2^n) \to \U(M)$. As noted at the end of \ref{subsec:duality}, $\U(2^n)$ is the free distributive lattice over the set of generators $n$. By the universal property of the free Heyting algebra, the lattice homomorphism $c^*$ has a unique extension to a Heyting homomorphism, $v$, from the free $n$-generated Heyting algebra to the Heyting algebra $\U(M)$, as in diagram (\ref{eq:semantics}).

\begin{equation}
\label{eq:semantics}
\begin{tikzpicture}[node distance=2cm, auto, baseline=(current  bounding  box.center)]
  \node (DL) {$F_{DL}(n)$};
  \node (is) [left of=DL,xshift=7mm] {$(2^n)^* =$};
  \node (HA) [above of=DL] {$F_{HA}(n)$};
  \node (UF) [right of=DL] {$\U(M)$};
  \draw[right hook->] (DL) to node {} (HA);
  \draw[->, dashed] (HA) to node {$v$} (UF);
  \draw[->] (DL) to node {$c^*$} (UF);
\end{tikzpicture}
\end{equation}
A point $x$ in a model $M$ is said to {\it satisfy} a formula $\varphi$ if, and only if, $x \in v(\varphi)$; we employ the usual notation: $M, x \models \varphi$. Note that, as an alternative to the above algebraic description, one may equivalently define the satisfaction relation for models by induction on the complexity of formulas; see e.g. \cite[Def. 2.1.8]{NBezh}. A model is said to satisfy $\varphi$ if every point of the model satisfies $\varphi$. A {\it p-morphism} $f$ from a model $M$ to a model $N$ is a p-morphism between the underlying frames of $M$ and $N$ which in addition satisfies, for any $x \in M$, $c_N(f(x)) = c_M(x)$. From the above definitions, it is clear that p-morphisms preserve truth, i.e., $M, x \models \varphi$ if, and only if, $N, f(x) \models \varphi$, for any formula $\varphi$. A {\it generated submodel} of $M$ is a submodel $M'$ such that the inclusion $f : M' \hookrightarrow M$ is a p-morphism, or equivalently, such that $M'$ is an up-set of $M$. We say $M'$ is a {\it p-morphic image} of $M$ if there exists a surjective p-morphism $f : M \twoheadrightarrow M'$.

Recall that a {\it general frame} is a tuple $(M,\leq,A)$, where $(M,\leq)$ is a poset and $A$ is a subalgebra of the Heyting algebra of up-sets of $M$. The elements of the algebra $A$ are called the {\it admissible sets} of the general frame. An important subclass of the class of general frames consists of the $(M,\leq,A)$ for which $(M,\leq)$ is the dual poset of the Heyting algebra $A$; these are precisely the {\it descriptive} general frames.\footnote{For an equivalent characterization of descriptive general frames as the `compact refined' general frames, cf. e.g. \cite[Def. 2.3.2, Thm. 2.4.2]{NBezh}.}

An {\it admissible colouring} on a general frame $(M,\leq,A)$ is a colouring $c : M \to 2^n$ with the additional property that, for each $1 \leq i \leq n$, the set $\{x \in M \ | \ c(x)_i = 1\}$ is admissible.
By the latter description and duality, admissible colourings $c$ on a descriptive frame $(M,\leq,A)$ correspond to homomorphisms $c^* : F_{DL}(n) \to A$. 
Note that, in this case, the semantics map $v$ defined in (\ref{eq:semantics}) also maps into $A$, since $A$ is a sub-Heyting-algebra of $\mathcal{U}(M)$.

We finally recall a few definitions and observations about so-called ``borders'' in Kripke models, that we will need in what follows.
\begin{definition}\label{def: borderP}
Let $M$ be a Kripke model.
\begin{enumerate}
\item If $A$ is an up-set in a Kripke frame $(M,\leq)$, then a \emph{border point of $A$} is a maximal element of the complement of $A$, i.e., a point $u$ which is not in $A$, while all its proper successors are in $A$. 
\item If $\varphi$ is a propositional formula, then a point $u$ is called a \emph{$\varphi$-border point} if $u$ is a border point of $v(\varphi)$, the subset of $M$ where $\varphi$ holds.
\item We say that $M$ is  a \emph{model with borders}, or that $M$ \emph{has borders}, if, for every $x \in M$ such that $x \not\models p$, there is a $p$-border point $u$ above $x$.
\end{enumerate}
\end{definition}
\begin{proposition}\label{prop:borders}
\quad
\begin{enumerate}
\item Every image-finite model has borders.
\item Every descriptive model has borders.
\end{enumerate}
\end{proposition}
\begin{proof}
Item (1) is straightforward. For the proof of (2) see, e.g., \cite[Thm. 2.3.24]{NBezh}.\qed
\end{proof}

\subsection{Canonical and universal models}\label{subsec:canunivmod}
The dual poset of the free $n$-generated Heyting algebra, $F_{HA}(n)$, is called the {\it canonical frame} and is denoted by $C(n)$. In logic terms, points in the canonical frame are so-called ``theories with the disjunction property''.
The canonical frame carries a natural colouring $c$, which is the dual of the inclusion $F_{DL}(n) \hookrightarrow F_{HA}(n)$. Concretely, $c(x)_i = 1$ if, and only if, the variable $p_i$ is an element of $x$. The model thus defined is called the {\it canonical model}, and is also denoted by $C(n)$.\footnote{The canonical frame and model are also known as the Henkin frame and model.} 

Note that, by the embedding $\eta : F_{HA}(n) \hookrightarrow \U(C(n))$, any element $\varphi$ of $F_{HA}(n)$ defines an up-set $\eta(\varphi) = \{x \in C(n) \ | \ \varphi \in x\}$ of $C(n)$. 
Since $\eta$ is in particular a Heyting homomorphism that extends $c^*$, it is equal to the semantics map $v$ for $C(n)$ defined in (\ref{eq:semantics}). Concretely, this means that, for any $x \in C(n)$ and $\varphi \in F_{HA}(n)$, we have $C(n), x \models \varphi$ if, and only if, $\varphi \in x$; this fact is often referred to as the {\it truth lemma}.

Let $\widehat{F_{HA}(n)}$ be the profinite completion of $F_{HA}(n)$; recall from \cite[Thm 4.7]{BezGehMinMor2006} that $\widehat{F_{HA}(n)}$ is the Heyting algebra of up-sets of $C(n)_{\mathrm{fin}} := \{x \in C(n) \ | \ {\uparrow}x \text{ is finite}\}$, the image-finite\footnote{Recall that a model $M$ is called \emph{image-finite} if, for each $w\in M$, the set of successors of $w$ is finite.} part of $C(n)$. The generated submodel $C(n)_{\mathrm{fin}}$ of $C(n)$ is known as the {\it universal model} and denoted by $U(n)$. 
\begin{lemma} \label{lem:univinjective}
The map $v : F_{HA}(n) \to \U(U(n))$ is injective.
\end{lemma}
\begin{proof}
Cf., e.g., \cite[Thm 3.2.20]{NBezh}.\qed
\end{proof}

Importantly, the universal model can be described by an inductive top-down construction, as follows.
\begin{theorem}\label{thm:unmodconst}
The universal model $U(n)$ is the unique image-finite model satisfying all of the following conditions:
\begin{enumerate}
\item there are $2^n$ maximal points with mutually distinct colours in $U(n)$;
\item for any $x \in U(n)$ and $c' < c(x)$, there is a unique point $x' \in U(n)$ with $c(x') = c'$ and ${\uparrow} x' = \{x'\} \cup {\uparrow} x$;
\item for any finite antichain $S \subseteq U(n)$ and $c' \leq \min\{c(x) \ | \ x \in S\}$, there is a unique point $x' \in U(n)$ with $c(x') = c'$ and ${\uparrow}x' = \{x'\} \cup \bigcup_{x \in S} {\uparrow} x$.
%\item any surjective p-morphism from $U(n)$ to a model $M$ is an isomorphism.
\end{enumerate}
\end{theorem}
\begin{proof}
Cf., e.g., \cite[Sec. 3.2]{NBezh}.\qed
\end{proof}
The following important fact states a `universal property' for the universal model. Following the usual terminology (cf., e.g., \cite[Sec 3.1]{NBezh}), the {\it depth} of a point $w$ in a frame $M$ is the maximal length of a chain in the generated subframe ${\uparrow} w$. We say that a frame $M$ has {\it finite depth} $\leq m$ if every chain in $M$ has size at most $m$.
\begin{proposition}\label{prop:findepmap}
If $M$ is a model on $n$ variables of finite depth $\leq m$, then there exists a unique p-morphism $f : M \to U(n)$. Moreover, the image of $f$ has depth $\leq m$.
\end{proposition}
\begin{proof}\footnote{This fact is well-known, cf. e.g. \cite[p. 428]{ChaZak1997}. We briefly recall the proof here. Also cf., e.g., \cite[Thm. 3.2.3]{FanYang}, for more details. Note, however, that we do not assume here that $M$ is finite, only that $M$ has finite depth.}
We prove the statement by induction on $m$. 
First let $M$ be a model of depth $0$. In this case, there is clearly a unique p-morphism from $M$ to $U(n)$, namely the one which sends each point in $M$ to the unique maximal point in $U(n)$ of the same colour. Now let $M$ be a model of depth $m + 1$, for $m \geq 0$.
Let $x \in M$ be arbitrary; we will define $f(x) \in U(n)$. Note that, for every $y > x$, the submodel $M_y := {\uparrow} y$ generated by $y$ has depth $\leq m$. Thus, for each $y > x$, let $f_y : M_y \to U(n)$ be the unique p-morphism; the image of $f_y$ has depth $\leq m$ by the induction hypothesis. Therefore, the set $S := \bigcup_{y > x}\mathrm{im}(f_y)$ has depth $\leq m$ in $U(n)$. If $S$ is empty, then $x$ is maximal, and we define $f(x)$ to be the unique maximal point of $U(n)$ that has the same colour as $x$. Otherwise, $S$ has finitely many minimal points, $s_0, \dots, s_k$, say. Pick points $y_0, \dots, y_k$ in $M$ such that $s_i \in \mathrm{im}(f_{y_i})$. If $k = 0$ and $c(y_0) = c(x)$, then we define $f(x) := s_0$. Otherwise, by Theorem~\ref{thm:unmodconst}, there is a unique point $s$ in $U(n)$ whose immediate successors are $s_0, \dots, s_k$ such that $c(s) = c(x)$; we define $f(x) := s$. It is straightforward to check that $f$ defined in this manner is the unique p-morphism from $M$ to $U(n)$, and clearly the image of $f$ has depth $\leq m + 1$.\qed
\end{proof}
\begin{remark}\label{rem:bisimilar}
Two points $x$ and $x'$ in a model $M$ of finite depth are bisimilar if, and only if, the unique p-morphism $f$ in Proposition~\ref{prop:findepmap} sends them to the same point of $U(n)$.
\end{remark}

\begin{definition}[De Jongh formulas]\label{def:Jongh}
We define formulas $\varphi_w$, $\psi_w$ and $\theta_w$, for each $w \in U(n)$, by induction on the depth of $w$.
Let $w \in U(n)$. Let $I_w$ denote the (finite) set of immediate successors of $w$.
By recursion, we assume that the formulas $\varphi_{w'}$, $\psi_{w'}$ and $\theta_{w'}$ have been defined for each $w' \in I_w$. We define:
\begin{align}
\theta_{w} &:= \bigvee_{w' \in I_w} \varphi_{w'},\\
\varphi_w &:= \bigwedge_{p \in T_w} p \wedge \bigwedge_{q \in B_w} (q \to \theta_w) \wedge \bigwedge_{w' \in I_w} \left(\psi_{w'} \to \theta_w\right), \label{deJF}\\
\psi_w &:= \varphi_w \to \theta_w,\label{deJF1}
\end{align}
where $T_w$ is the set of propositional variables $p$ which are true in $w$, $B_w$ is the set of propositional variables $q$ such that $w$ is a $q$-border point.\footnote{We use the usual convention that $\bigvee \emptyset = \bot$ and $\bigwedge \emptyset = \top$.}
\end{definition}
Note that the above definition includes the case where $w$ is a maximal point, i.e., $k = 0$. Also note that the syntactic shape of our definition of $\varphi_w$ is slightly different from the usual definition (e.g. \cite[Def. 3.3.1]{NBezh}), but easily seen to be equivalent using the fact that $(\bigvee_{i=1}^m \alpha_i) \to \beta$ is equivalent in IPC to $\bigwedge_{i=1}^m (\alpha_i \to \beta)$, for any formulas $\alpha_1,\dots,\alpha_m$ and $\beta$.
The following theorem shows which subsets of the universal model are defined by De Jongh formulas.
\begin{theorem}\label{thm: Jongh}
For each $w\in U(n)$, we have  $v(\theta_w) = ({\uparrow}w) \setminus \{w\}$, $v(\varphi_w) = {\uparrow}w$,
and $v(\psi_w) = U(n)\setminus {\downarrow}w$.
\end{theorem}

\begin{proof}
By induction on the depth of $w$, cf., e.g., \cite[Thm.\ 3.3.2]{NBezh}.\qed
\end{proof}

Note that the de Jongh formula $\psi_w$ has the following property: a frame $G$ refutes
$\psi_w$ iff there is a generated subframe of $G$ p-morphically mapped onto the subframe of $U(n)$ generated
by $w$. In this way, de Jongh formulas correspond to the so-called Jankov or splitting formulas,
see \cite[Sec.\ 3.3]{NBezh} for the details.  

\section{Separated points and the meet-implication fragment}\label{sec:sepmeetimp}

In this section we use a duality for Heyting algebras and $(\wedge, \to)$-homomorphisms for characterizing 
$n$-universal models of the $(\wedge, \to)$-fragment of IPC (Theorem~\ref{thm:diegostrong}) and for characterizing $(\wedge, \to)$-definable up-sets of 
$n$-universal models of IPC (Theorem~\ref{thm:characterization}). The main technical contribution is the characterization of the dual model of the 
$(\wedge, \to)$-subalgebra of a Heyting algebra generated by a finite set of generators (Theorem~\ref{thm:gensubdual}).
Our proofs rely on discrete duality and do not use topology.  
They can be extended to Priestley \cite{Pri1970} and Esakia \cite{Esak} dualities by adding topology, but we will not use this (explicitly) in this paper. 

In the study of the meet-implication fragment, the following notion of `separated point' in a model will be crucial.\footnote{This notion has it roots in \cite{RenHenJon}. Our `separated' points are precisely those points which are `not inductive and not full' in the terminology of \cite[Def. 5]{RenHenJon}.}
\begin{definition}\label{dfn:separated}
Let $M$ be a model. A point $x \in M$ is \emph{separated} if, and only if, there exists a propositional variable $q$ for which $x$ is a $q$-border point.
\end{definition}
The following easy lemma will be used frequently in what follows.
\begin{lemma}\label{lem:pressep}
Let $f : M \to N$ be a p-morphism between models. If $x$ is a separated point in $M$, then $f(x)$ is separated in $N$.
\end{lemma}
\begin{proof}
Let $x \in M$ be a separated point. Choose $q$ such that $x$ is a $q$-border point. We claim that $f(x)$ is a $q$-border point in $N$, and therefore separated. Indeed, $N, f(x) \not\models q$ since $f$ preserves colourings. Also, if $y' > f(x)$, then since $f$ is a p-morphism we may pick $x' > x$ such that $f(x') = y'$. Since $x' > x$, we have that $M, x' \models q$ since $x$ is a $q$-border point, so $q$ also holds in $y = f(x')$, since $f$ preserves colourings.\qed
\end{proof}
The following alternative characterization of separated points relates them to the $(\wedge,\to)$-fragment.

\begin{lemma}\label{lem: 8}
Let $x$ be a point in a model $M$. The following are equivalent:
\begin{enumerate}
\item the point $x$ is separated;
\item there exists a $(\wedge,\to)$-formula $\varphi$ such that $x$ is a $\varphi$-border point.
\end{enumerate}
\end{lemma}
\begin{proof}
It is clear that (1) implies (2). For (2) implies (1), we prove the contrapositive. Suppose that $x$ is not separated. We prove the negation of (2), i.e., $x$ is not a $\varphi$-border point for any $(\wedge,\to)$-formula $\varphi$, by induction on complexity of $\varphi$. For $\varphi$ a propositional variable, this is true by assumption. For $\varphi = \psi \wedge \chi$, note that $v(\varphi)^c = v(\psi)^c \cup v(\chi)^c$. From this equality, it follows that if $x$ were a $\varphi$-border point, it would already be either a $\psi$-border point or a $\chi$-border point, which contradicts the induction hypothesis. For $\varphi = \psi \to \chi$, suppose that $x$ is a $\varphi$-border point. We will prove that $x$ is also a $\chi$-border point, which again contradicts the induction hypothesis. By maximality of $x$, all $y > x$ satisfy $\psi \to \chi$. However, $x$ does not satisfy $\psi \to \chi$, so we must have that $x \in v(\psi) \cap v(\chi)^c$. Since $v(\psi)$ is an up-set, we conclude that, for all $y > x$, $y \in v(\psi)$, and therefore $y \in v(\chi)$. Hence, $x$ is a $\chi$-border point, as required.\qed
\end{proof}
For a model $M$, we denote by $M^s$ the submodel consisting of the separated points of $M$. That is, the order and colouring on $M^s$ are the restrictions of the corresponding structures on $M$. (Note that the model $M^s$ is a submodel, but almost never a {\it generated} submodel, i.e.\ an up-set, of $M$!) 
\begin{lemma}\label{lem:finiteheight}
Let $M$ be a model on $n$ variables. The submodel $M^s$ has finite depth $\leq n$.
\end{lemma}
\begin{proof}
Let $C$ be a chain in $M^s$. For any $x, y \in M^s$, if $x < y$, then $c(x) < c(y)$, since $x$ is separated and $y > x$ in $M$. Therefore, $\{c(x) \ | \ x \in C\}$ is a chain in the poset $(2^n, \leq)$, so that it must have size $\leq n$. Hence, $C$ has size at most $n$. \qed
\end{proof}
\begin{definition}[The model $M_{\wedge,\to}$]\label{def:13}
Let $M$ be a model and $M^s$ its submodel of separated points. Let $f : M^s \to U(n)$ be the unique p-morphism which exists by Lemma~\ref{lem:finiteheight} and Proposition~\ref{prop:findepmap}. Define $M_{\wedge,\to} := \mathrm{im}(f)$ to be the generated submodel of $U(n)$ consisting of those points in the image of $f$, as in the following diagram. %The construction of $M_{\wedge,\to}$ from a model $M$ is depicted in diagram (\ref{eq:Mconstruction}).
\begin{equation}
\label{eq:Mconstruction}
\begin{tikzpicture}[node distance=1.9cm, auto, baseline=(current  bounding  box.center)]
  \node (Ms) {$M \supseteq M^s$};
  \node (Mwt) [right of=Ms] {$M_{\wedge,\to}$};
  \node (Un) [above of=Mwt] {$U(n)$};

  \draw[->] (Ms) to node[left,yshift=2mm] {$f$} (Un);
  \draw[->>] (Ms) to (Mwt);
  \draw[right hook->] (Mwt) to (Un);
\end{tikzpicture}
\end{equation}
\end{definition}
The above definition can in particular be applied to $U(n)$ itself. The following proposition characterizes the points in the generated submodel $U(n)_{\wedge,\to}$ of $U(n)$.
\begin{proposition}\label{prop:pointsUnwt}
Let $n \geq 1$. The generated submodel $U(n)_{\wedge,\to}$ consists exactly of those points $x \in U(n)$ such that for all $y \geq x$, $y$ is separated.
\end{proposition}
\begin{proof}
Write $S := \{x \in U(n) \ | \ \text{ for all } y \in U(n), \text{ if } y \geq x, \text{ then } y \text{ is separated}\}$.
Note that $S$ is a generated submodel of $U(n)$, and also of $U(n)^s$.
Let $f$ be as in Definition~\ref{def:13}. By definition, $U(n)_{\wedge,\to} = \mathrm{im}(f)$. 
We show that $S = \mathrm{im}(f)$.
If $x' \in U(n)^s$, then $f(x') \in S$: for any $y \in U(n)$ with $y \geq x$, there exists $y' \in U(n)^s$ such that $f(y') = y$. By Lemma~\ref{lem:pressep}, $y$ is separated. 
For the converse, note that the restriction, $g$, of $f$ to the generated submodel $S$ is still a p-morphism, since $S$ is a generated submodel of $U(n)^s$.
Also, the inclusion map $i : S \to U(n)$ is a p-morphism, since $S$ is a generated submodel of $U(n)$. 
Therefore, by the uniqueness part of Proposition~\ref{prop:findepmap}, we must have $i = g$.
Thus, if $x$ is in $S$, then $x = i(x) = g(x) = f(x)$. In particular, $x$ is in $\mathrm{im}(f)$.\qed
\end{proof}
\begin{lemma}\label{lem:Mwt-sep}
For any model $M$, $M_{\wedge,\to}$ is contained in $U(n)_{\wedge,\to}$. In particular, $M_{\wedge,\to}$ is a generated submodel of $U(n)$ of depth $\leq n$, and thereby a finite model.
\end{lemma}
\begin{proof}
Since $M_{\wedge,\to}$ is the image of a p-morphism, it is a generated submodel, and all its points are separated, so by Proposition~\ref{prop:pointsUnwt}, every point of $M_{\wedge,\to}$ is in $U(n)_{\wedge,\to}$.
The `in particular'-part follows from Lemma~\ref{lem:finiteheight}.\qed
\end{proof}

In Theorem~\ref{thm:gensubdual} below, we will show that, for any model with borders $M$, the model $M_{\wedge,\to}$ is dual to the $(\wedge,\to)$-subalgebra of $A$ that is generated by the admissible up-sets $v(p_1),\dots,v(p_n)$. We need two lemmas, Lemma~\ref{lem:frob} and Lemma~\ref{lem:key}.
\begin{lemma}\label{lem:frob}
Let $f : H \to K$ be a function between Heyting algebras with an upper adjoint $g : K \to H$. Then $f$ preserves binary meets if, and only if, for all $a \in H$, $b \in K$, the equality $a \to g(b) = g(f(a) \to b)$ holds.
In particular, if $f$ is surjective and preserves binary meets, then $g$ preserves Heyting implication.
\end{lemma}
\begin{proof}
Let $a \in H$ be arbitrary. Consider the following two diagrams.
\begin{equation}
\label{eq:meetpres}
\begin{tikzpicture}[node distance=2cm, auto, baseline=(current  bounding  box.center)]
  \node (H1) {$H$};
  \node (H2) [below of=H1] {$H$};
  \node (K1) [right of=H1] {$K$};
  \node (K2) [right of=H2] {$K$};
  \draw[->] (H1) to node {$f$} (K1);
  \draw[->] (H1) to node[left] {$a \wedge -$} (H2);
  \draw[->] (H2) to node {$f$} (K2);
  \draw[->] (K1) to node {$f(a) \wedge -$} (K2);
\end{tikzpicture}
\hspace{10mm}
\begin{tikzpicture}[node distance=2cm, auto, baseline=(current  bounding  box.center)]
  \node (H1) {$H$};
  \node (H2) [below of=H1] {$H$};
  \node (K1) [right of=H1] {$K$};
  \node (K2) [right of=H2] {$K$};
  \draw[->] (K1) to node[above] {$g$} (H1);
  \draw[->] (H2) to node[left] {$a \to -$} (H1);
  \draw[->] (K2) to node[above] {$g$} (H2);
  \draw[->] (K2) to node[right] {$f(a) \to -$} (K1);
\end{tikzpicture}
\end{equation}
A way to express the assertion that $f$ preserves binary meets is that, for all $a \in H$, the left diagram in (\ref{eq:meetpres}) commutes. By uniqueness of adjoints, the left diagram in (\ref{eq:meetpres}) commutes if, and only if, the right diagram in (\ref{eq:meetpres}) commutes.

The `in particular'-part now follows since, if $f$ is surjective, then $b' = fg(b')$ for any $b' \in K$ (Proposition~\ref{prop:adjbasicfacts}).\qed
\end{proof}
Lemma~\ref{lem:frob} and its proof are very similar to, and were in fact directly inspired by, the {\it Frobenius condition} in \cite[Def. p. 157]{Pit1983} and the remark following it; we leave further exploration of the precise connection to future research.

The following lemma now provides the key connection between the construction of $M^s$ and the $(\wedge,\to)$-fragment.
\begin{lemma}\label{lem:key}
Let $M$ be a model with borders. 
Consider the following diagram:
\begin{equation}\label{eq:keylemdiagram}
\begin{tikzpicture}[node distance=2cm, auto, baseline=(current  bounding  box.center)]
  \node (FHA) {$F_{HA}(n)$};
  \node (Fwt) [left of=FHA] {$F_{\wedge,\to}(n)$};
  \node (M) [below of=FHA] {$\U(M)$};
  \node (Ms) [right of=M] {$\U(M^s)$};
  \draw[->] (FHA) to node {$v$} (M);
  \draw[->] (FHA) to node {$v^s$} (Ms);
  \draw[transform canvas={yshift=.5ex}, left hook->] (Ms) to node[above] {$r$} 
%node[below, xshift={1.5ex}, yshift={-1ex},rotate=-90] {$\vdash$} 
(M);
    \draw[transform canvas={yshift=-.5ex}, ->>] (M) to node[below] {$q$} (Ms);
    \draw[right hook->] (Fwt) to node[above]  {$i$} (FHA);
\end{tikzpicture}
\end{equation}
where $i$ is the natural inclusion, $v$ and $v^s$ are the valuation maps of $M$ and $M^s$, respectively, $q$ is the lattice homomorphism dual to the inclusion $M^s \hookrightarrow M$, and $r$ is its upper adjoint. Then $r$ is a $(\wedge,\to)$-homomorphism, and
\begin{equation}\label{eq:toprove2}
r \circ v^s \circ i = v \circ i.
\end{equation}
\end{lemma}
\begin{remark}
Note that, by Proposition~\ref{prop:adjbasicfacts}, the function $r$ sends an up-set $V$ of $M^s$ to the up-set $\{x \in M \ | \ \forall y \geq x\,(y \in M^s \Rightarrow y \in V)\}$ of $M$. Therefore, the equality (\ref{eq:toprove2}) says precisely that, for any $(\wedge,\to)$-formula $\varphi$ and $x \in M$, we have
\begin{equation}\label{eq:toprove}
M, x \models \varphi \iff \forall y \geq x\,(y \in M^s \Rightarrow M^s, y \models \varphi).
\end{equation}
In this sense, Lemma~\ref{lem:key} is an algebraic rendering of the crucial ingredient to \cite[Proof of Thm. 1]{RenHenJon}. The proof we give here is different in spirit.
\end{remark}

\begin{proof}[of Lemma~\ref{lem:key}]
Note that $r$ is $\wedge$-preserving since it is an upper adjoint, and $r$ is $\to$-preserving by Lemma~\ref{lem:frob}. Therefore, both $v \circ i$ and $r \circ v^s \circ i$ are $(\wedge,\to)$-homomorphisms. 
Hence, to prove (\ref{eq:toprove2}), it suffices to prove that $v \circ i$ and $r \circ v^s \circ i$ are equal on propositional variables. 
Let $p$ be any propositional variable. We have that $vi(p) \leq rqvi(p) = rv^si(p)$, because $r$ is upper adjoint to $q$ and $v^s(p) = qv(p)$ by definition of $v^s$.
On the other hand, suppose that $x \not\in vi(p)$. Since $M$ is a model with borders, pick $y \in \max(v(p)^c)$ such that $y \geq x$. Then $y \in M^s$, so $x \not\in rqvi(p)$, as required.\qed
\end{proof}

\begin{proposition}\label{cor: 17}
Let $M$ be a model with borders and let $M^s$ and $f:M^s\to M_{\wedge, \to}$ be as in Definition~\ref{def:13}. For any $w\in M^s$ and $(\wedge, \to)$-formula $\varphi$, we have 
\begin{equation}
M, w\models \varphi \iff M_{\wedge, \to}, f(w)\models \varphi.
\end{equation}
\end{proposition}
\begin{proof}
Immediate from Definition~\ref{def:13} and the equivalence in (\ref{eq:toprove}).\qed
\end{proof}

The above considerations in particular allow us to prove the following theorem, originally due to Diego \cite{Diego}.
\begin{theorem}[Diego]\label{thm:diegoweak}
For any $n$, $F_{\wedge,\to}(n)$ embeds as a $(\wedge,\to)$-subalgebra into $\U(U(n)_{\wedge,\to})$.
In particular, $F_{\wedge,\to}(n)$ is finite and therefore, the variety of implicative meet-semilattices is locally finite. 
\end{theorem}
\begin{proof} 
Let $h : F_{\wedge,\to}(n) \to \U(U(n)_{\wedge,\to})$ be the extension of the assignment $p_i \mapsto p_i$ to a $(\wedge,\to)$-homomorphism. We show that $h$ is injective. Suppose that $\varphi \neq \psi$ in $F_{\wedge,\to}(n)$. By Lemma~\ref{lem:univinjective}, we have $v^{U(n)}(i(\varphi)) \neq v^{U(n)}(i(\psi))$. By Lemma~\ref{lem:key}, applied to $U(n)$, we have $v^s(i(\varphi)) \neq v^s(i(\psi))$, since $r$ is injective. This means that there exists $x \in U(n)^s$ such that $U(n)^s, x \models \varphi$ and $U(n)^s, x \not\models \psi$. Hence, since $f : U(n)^s \to U(n)$ is a p-morphism, we obtain $f(x) \in h(\varphi)$ and $f(x) \not\in h(\psi)$, so $h(\varphi) \neq h(\psi)$, as required.
The `in particular'-part now follows, since by Lemma~\ref{lem:Mwt-sep}, $U(n)_{\wedge,\to}$ is finite.\qed
\end{proof}

It follows from Theorem~\ref{thm:diegoweak} that 
$F_{\wedge,\to}(n)$ is a finite Heyting algebra for each $n$, in which the binary supremum is given by
\begin{equation}
 \varphi \veebar \psi = \bigwedge \{ \chi \in F_{\wedge,\to}(n) \ | \ \varphi \leq \chi \text{ and } \psi \leq \chi \},
\end{equation}
and the bottom element $\underline{\bot}$ is given by $p_1 \wedge \cdots \wedge p_n$.
\begin{definition}\label{def:s}
For each intuitionistic formula $\varphi$, let $s(\varphi)$ denote the formula obtained from $\varphi$ by replacing each occurrence of a disjunction $\vee$ by $\veebar$, and replacing each occurrence of $\bot$ by $\underline{\bot}$.
\end{definition}
Algebraically, the above definition is the unique Heyting algebra homomorphism $s : F_{HA}(n) \to F_{\wedge,\to}(n)$ extending the assignment $p_i \mapsto p_i$. 
This means that if $\varphi$ is provable in IPC, $s(\varphi)$ is also provable in IPC.
In particular, if 
$\varphi$ implies $\psi$ in IPC, then $s(\varphi\to \psi) = s(\varphi)\to s(\psi)$  is also provable in IPC. This means that  $s(\varphi)$ implies $s(\psi)$. 

\begin{theorem}\label{thm:definable}
Every up-set of $U(n)_{\wedge,\to}$ is definable by a $(\wedge,\to)$-formula.
\end{theorem}
\begin{proof}
Let $U$ be an up-set of $U(n)_{\wedge,\to}$. Recall that $U(n)_{\wedge,\to}$ is a finite generated submodel of $U(n)$, by Lemma~\ref{lem:Mwt-sep}. We denote by $\min(U)$ the finite set of minimal points of $U$. It follows from Theorem~\ref{thm: Jongh} that $U$ is defined by the disjunction $\varphi_U := \bigvee_{u \in \min(U)} \varphi_u$ of de Jongh formulas. We also have the $(\wedge,\to)$-formula $s(\varphi_U)$ defined as in Definition~\ref{def:s}.
To prove the theorem, it therefore suffices to prove the following claim.

\vspace{2mm}

{\bf Claim.} The up-set of $U(n)_{\wedge,\to}$ defined by $s(\varphi_U)$ is equal to $U$. 

\vspace{2mm}

{\it Proof of Claim.} By induction on the partial order of inclusion of up-sets of $U(n)_{\wedge,\to}$. 
For the base case, $U = \emptyset$, note that $\min(\emptyset) = \emptyset$, so that $s(\varphi_\emptyset) = s(\bot) = p_1 \wedge \cdots \wedge p_n$, which indeed defines the empty subset of $U(n)_{\wedge,\to}$, since no separated point makes all propositional variables true.

Now suppose that $U$ is a non-empty up-set in $U(n)_{\wedge,\to}$. The induction hypothesis is that, for all proper subsets $V \subsetneq U$, the formula $s(\varphi_V)$ defines $V$. 

We distinguish two cases: (1) $U$ has a one minimal point; (2) $U$ has more than one minimal point.

\vspace{2mm}

(1) Let $w$ be the minimum of $U$. By the induction hypothesis, for every $w' \in I_w$, $s(\varphi_{w'})$ defines ${\uparrow} w'$ in $U(n)_{\wedge,\to}$. Therefore, the formula $s(\psi_{w'}) = s(\varphi_{w'}) \to s(\theta_{w'})$ defines $({\downarrow} w')^c$  in $U(n)_{\wedge,\to}$.  Thus, $s(\psi_{w'})$ and $\psi_{w'}$ define the same up-set in $U(n)_{\wedge,\to}$. Moreover, the induction hypothesis also implies that $s(\theta_{w}) = s(\varphi_{{\uparrow}I_w})$ defines ${\uparrow} I_{w}$  in $U(n)_{\wedge,\to}$. Thus, the formulas $s(\theta_w)$ and $\theta_w$ define the same subset of $U(n)_{\wedge,\to}$. It follows that a point $x \in U(n)_{\wedge,\to}$ satisfies $s(\varphi_w)$ if, and only if, $x$ satisfies $\varphi_w$. By Theorem~\ref{thm: Jongh}, the latter holds if, and only if, $x \geq w$. Thus, $s(\varphi_w)$ defines ${\uparrow} w = U$.

\vspace{2mm}

(2) Note first that, if $u \in U$, then $u \geq w$ for some $w \in \min(U)$. Therefore, $u \models s(\varphi_w)$, using case (1). Since $\varphi_w$ implies $\varphi_U$ in IPC, we have that $s(\varphi_w)$ implies $s(\varphi_U)$. Hence, $u \models s(\varphi_U)$.
It remains to show that there is no border point $u$ of $U$ which satisfies $s(\varphi_U)$. Let $u$ be a border point of $U$. We write $B$ for the up-set ${\uparrow} I_u$, which is a subset of $U$ since $u$ is a border point of $U$. We will distinguish two sub-cases: (a) $B = U$, and (b) $B\subsetneq U$. 

\vspace{2mm}

(a) $B = U$. Then, in particular, $I_u = \min(B) = \min(U)$.
Since $u$ is separated, choose a propositional variable $q$ so that $u$ is a $q$-border point. Then every point $w \in \min(U) = I_u$ satisfies $q$, so $\varphi_U$ implies $q$ in IPC, so $s(\varphi_U)$ implies $s(q) = q$ in IPC. Since $u$ does not satisfy $q$, $u$ also does not satisfy $s(\varphi_U)$.
\vspace{2mm}

(b) $B \subsetneq U$. 
Applying the induction hypothesis to $B$, we see that $s(\varphi_B) = \underline{\bigvee}_{u' \in I_u} s(\varphi_{u'})$ defines $B$. It follows from this that $u$ does not satisfy $s(\psi_u)$, since $u$ certainly satisfies $s(\varphi_u)$, using the induction hypothesis again.
An easy application of Theorem~\ref{thm: Jongh} shows that, for every $w \in \min(U)$, $\varphi_{w}$ implies $\psi_u$ in IPC, since $w \nleq u$. Hence, $s(\varphi_{w})$ implies $s(\psi_u)$, for every $w \in I_u$. Therefore, $s(\varphi_U)$ implies $s(\psi_u)$. However, $u$ does not satisfy $s(\psi_u)$, so $u$ does not satisfy $s(\varphi_U)$.
\qed
\end{proof}

Let $M$ be a model with borders.
In diagram (\ref{eq:thmdiagram}) below we show how the valuation of formulas in the models $M$ and $M^s$, as in diagram~(\ref{eq:keylemdiagram}), is related to the unique map $f : M^s \to U(n)$ that was used in the diagram (\ref{eq:Mconstruction}).
\begin{equation}\label{eq:thmdiagram}
\begin{tikzpicture}[node distance=2cm, auto, baseline=(current  bounding  box.center)]
  \node (FHA) {$F_{HA}(n)$};
  \node (Fwt) [left of=FHA] {$F_{\wedge,\to}(n)$};
  \node (M) [below of=FHA] {$\U(M)$};
  \node (Ms) [right of=M] {$\U(M^s)$};
  \node (UMwt) [right of = Ms] {$\U(M_{\wedge,\to})$};
  \node (UUn) [above of = UMwt] {$\U(U(n))$};
  \draw[->] (FHA) to node {$v$} (M);
  \draw[->] (FHA) to node {$v^s$} (Ms);
  \draw[transform canvas={yshift=.5ex}, left hook->] (Ms) to node[above] {$r$} 
%node[below, xshift={1.5ex}, yshift={-1ex},rotate=-90] {$\vdash$} 
(M);
    \draw[transform canvas={yshift=-.5ex}, ->>] (M) to node[below] {$q$} (Ms);
    \draw[right hook->] (Fwt) to node[above]  {$i$} (FHA);
    \draw [left hook->] (UMwt) to node[above] {$h$} (Ms);
    \draw[->] (UUn) to node[left] {$f^*$} (Ms);
    \draw[->] (FHA) to node[above] {$v^{U(n)}$} (UUn);
    \draw[->>] (UUn) to node[right] {$t$} (UMwt);
\end{tikzpicture}
\end{equation}
In the above diagram, the left part of the diagram is defined as in (\ref{eq:keylemdiagram}), $v^{U(n)}$ denotes the natural valuation on $U(n)$, and the triangle $f^* = h \circ t$ is the dual of the triangle in (\ref{eq:Mconstruction}).

\begin{theorem}\label{thm:gensubdual}
Let $M$ be a model with borders. Denote by $B$ be the $(\wedge,\to)$-subalgebra of $\mathcal{U}(M)$ that is generated by $v(p_1), \dots, v(p_n)$. Then $B$ is equal to the image of the composite $rh$. In particular, $B$ is isomorphic to the implicative meet-semilattice $\U(M_{\wedge,\to})$.
\end{theorem}
\begin{proof}
Chasing the diagram (\ref{eq:thmdiagram}), we have:
\begin{equation}\label{eq:diagramchase}
vi = rv^si = r f^* v^{U(n)} i = rhtv^{U(n)}i,
\end{equation}
where we use Lemma~\ref{lem:key} and the fact that $v^s = f^* v^{U(n)}$, since $f$ is a p-morphism of models. 
Note that $B = \mathrm{im}(vi)$, so we need to show that $\mathrm{im}(rh) = \mathrm{im}(vi)$.
For the inclusion ``$\subseteq$'', let $U \in \mathcal{U}(M_{\wedge,\to})$, and $V := rh(U)$; we prove that $V \in \mathrm{im}(vi)$. Since $U$ is an up-set in $U(n)_{\wedge,\to}$ by Lemma~\ref{lem:Mwt-sep}, Theorem~\ref{thm:definable} implies that there is a $(\wedge,\to)$-formula $\varphi$ such that $v^{U(n)}i(\varphi) \cap U(n)_{\wedge,\to} = U$. Therefore, since $U \subseteq M_{\wedge,\to}$, we have 
$tv^{U(n)}i(\varphi) = v^{U(n)}i(\varphi) \cap M_{\wedge, \to} = U$. Thus, $V = rh(U) = rhtv^{U(n)}i(\varphi) = vi(\varphi)$, using (\ref{eq:diagramchase}), so $V \in \mathrm{im}(vi)$.

For the inclusion ``$\supseteq$'', note first that $\mathrm{im}(r h)$ contains $v(p_1),\dots,v(p_n)$. It thus remains to show that $\mathrm{im}(r h)$ is a $(\wedge,\to)$-subalgebra of $\U(M)$, or equivalently, that $r h$ preserves $\wedge$ and $\to$. Since $r$ is an upper adjoint and $h$ is a Heyting homomorphism, $r h$ preserves $\wedge$. Moreover, using Lemma~\ref{lem:frob} and the fact that $qr = \id$, we have, for any $U, V \in \U(M_{\wedge,\to})$, that
\[ rh(U) \rightarrow rh(V) = r(qrh(U) \rightarrow h(V)) = r(h(U) \rightarrow h(V)) = rh(U \rightarrow V),\]
where the last step uses that $h$ is a Heyting homomorphism.\qed
\end{proof}

We now use this theorem to prove three facts about the $(\wedge,\to)$-fragment of IPC. The first is a strong form of Diego's theorem.

\begin{corollary}\label{thm:diegostrong}
For any $n$, $F_{\wedge,\to}(n) \cong \U(U(n)_{\wedge,\to})$.
\end{corollary}
\begin{proof}
Apply Theorem~\ref{thm:gensubdual} to the model $U(n)$. Using Lemma~\ref{lem:univinjective}, the map $v^{U(n)}i : F_{\wedge,\to}(n) \to \mathcal{U}(U(n))$ is injective, so $F_{\wedge,\to}(n)$ is isomorphic to the image of $v^{U(n)}i$. The image of $v^{U(n)} i$ is the subalgebra generated by $v(p_1),\dots,v(p_n)$, which, by Theorem~\ref{thm:gensubdual} is isomorphic to $\U(U(n))_{\wedge,\to}$.\qed
\end{proof}

\begin{theorem}\label{thm:sepform}
For any $\varphi \in F_{HA}(n)$
 and any model $M$ and $x \in M^s$, we have:
\[ M^s, x \models \varphi \iff M^s, x \models s(\varphi).\]
\end{theorem}
\begin{proof}
Recall that $s$ is the unique Heyting homomorphism $F_{HA}(n) \to F_{\wedge,\to}(n)$ such that $s(p) = p$ for all propositional variables $p$. Note that $s i$ is the identity on $F_{\wedge,\to}(n)$, so $s$ is surjective. Also note that $tv^{U(n)}i$ is surjective, as we showed in the proof of the inclusion ``$\subseteq$'' of Theorem~\ref{thm:gensubdual}. We conclude that $tv^{U(n)}is$ is a surjective $(\wedge,\to)$-preserving map, and therefore it is a Heyting homomorphism by Lemma~\ref{lem:surjheyt}. Now, $htv^{U(n)}is$ is also a Heyting homomorphism and $htv^{U(n)}is(p) = htv^{U(n)}(p) = v^s(p)$. By uniqueness of the map $v^s$, we conclude that $htv^{U(n)}is = v^s$. Thus, for any $x \in M^s$, we have 
\[ x \in v^s(\varphi) \iff x \in htv^{U(n)}is(\varphi) \iff x \in f^*v^{U(n)}is(\varphi) \iff x \in v^sis(\varphi),\]
as required. \qed
\end{proof}

\begin{theorem}\label{thm:characterization}
Let $M$ be a model with borders.
Let $U \subseteq M$ be an up-set. The following are equivalent:
\begin{enumerate}
\item There exists a $(\wedge,\to)$-formula $\varphi$ such that $v(\varphi) = U$;
\item For all $x \in M$, if, for all $z \in M^s$ such that $z \geq x$, there exists $y \in U \cap M^s$ bisimilar to $z$ in $M^s$, then $x \in U$;
\item For all $x \in M$, 
\begin{itemize}
\item[$(a)$] if all separated points above $x$ are in $U$, then $x \in U$, and
\item[$(b)$] %$U \cap M^s$ is saturated with respect to bisimilarity in $M^s$, i.e., 
if $x \in M^s$ and there exists $x' \in U \cap M^s$ which is bisimilar to $x$ in $M^s$, then $x \in U$.
\end{itemize}
\end{enumerate}
\end{theorem}
\begin{proof}
By Theorem~\ref{thm:gensubdual}(1), the up-sets which are definable by a $(\wedge,\to)$-formula are precisely the up-sets in the image of $rh$. Let $h^\flat$ denote the lower adjoint of $h$, which is given explicitly by sending $S \in \U(M^s)$ to $f(S) \in \U(M_{\wedge,\to})$. By Proposition~\ref{prop:adjbasicfacts}(1), applied to the adjunction $h^\flat q \dashv rh$, an up-set $U$ is in $\mathrm{im}(rh)$ if, and only if, $r h h^\flat q(U) \subseteq U$. Writing out the definitions of $r$, $h$, $h^\flat$ and $q$, we see that this condition is equivalent to:
\[ \forall x \in M, \text{ if } \left(\forall z \in M^s \text{ if } z \geq x \text{ then } z \in f^{-1}(f(U \cap M^s))\right) \text{ then } x \in U.\]
This condition is in turn equivalent to (2), using Remark~\ref{rem:bisimilar}. If (2) holds, then (3a) is clear. For (3b), suppose $x$ is separated and there exists $x' \in U \cap M^s$ which is bisimilar to $x$ in $M^s$. By bisimilarity, for any $z \in M^s$ with $z \geq x$, there exists $y \in M^s$ with $y \geq x'$ and $y$ bisimilar to $z$ in $M^s$. Moreover, since $U$ is an up-set containing $x'$, we have $y \in U$. Using (2), we conclude that $x \in U$. Now assume (3) and let $x \in M$ be a point such that for all $z \in M^s$ with $z \geq x$, there exists $y \in U \cap M^s$ bisimilar to $z$ in $M^s$. If $z$ is any separated point above $x$, then it follows from applying (3b) to $z$ that $z \in U$. Therefore, by (3a), $x \in U$.
\qed
\end{proof}

\section{Subframe formulas and uniform interpolation}\label{sec:subframeunifint}
In this section we will apply the results obtained in the previous section 
to show that 
$(\wedge, \to)$-versions of
de Jongh formulas correspond to  subframe formulas 
in just the same way as de Jongh formulas correspond to 
Jankov formulas (Theorem~\ref{thm: 5}). 
We will also use the characterization of $(\wedge,\to)$-definable up-sets of $U(n)$ 
to prove that uniform interpolants in the $(\wedge, \to)$-fragment of IPC are not always given by the IPC-uniform interpolants (Example~\ref{exa:unifint}). \\
We need an auxiliary lemma before proving the main theorem of this section.
\begin{lemma}\label{lem: subf}
For each finite rooted frame $F$, there exist $n \in \omega$ and a colouring $c : F \to 2^n$ such that $M= (F,c)$ is 
isomorphic to 
a generated submodel 
of $U(n)_{\wedge,\to}$.
\end{lemma} 
\begin{proof}
Let $n := |F|$ and enumerate the points of $F$ as $x_1, \dots, x_n$. Define $c(x_i)_j$, the $j^\mathrm{th}$ coordinate of the colour of the point $x_i$, to be $1$ if $x_i \geq x_j$, and $0$ otherwise. All points in $M = (F,c)$ have distinct colours, and are in particular separated, so $M = M^s$. Let $f$ be the unique p-morphism from $M = M^s$  to $U(n)$ from Proposition~\ref{prop:findepmap}, its image is $M_{\wedge,\to}$. Recall from Lemma~\ref{lem:Mwt-sep} that $M_{\wedge,\to}$ is a submodel of $U(n)^s$. Let $g$ be the unique p-morphism from $U(n)^s$ onto $U(n)_{\wedge,\to}$. Since the composite $gf : M \to U(n)_{\wedge,\to}$ preserves colours, it is injective, and it is therefore an isomorphism onto a generated submodel of $U(n)_{\wedge,\to}$.
\qed
\end{proof}

\begin{theorem}\label{thm: 5}
Let $F$ be a finite rooted frame and let $M = (F,c)$ be the model on $F$ defined in the proof of Lemma~\ref{lem: subf}. There exists a $(\wedge,\to)$-formula $\beta(F)$ such that for any descriptive model $N$ we have 
$$
N \not\models \beta(F)\ \ \iff \ \ M\ \mbox{is a p-morphic image of $N^s$.}
$$
\end{theorem}
\begin{proof}
By Lemma~\ref{lem: subf}, $M$ is isomorphic to a generated submodel of $U(n)_{\wedge,\to}$. Without loss of generality, we will assume in the rest of this proof that $M$ actually {\it is} a generated submodel of $U(n)_{\wedge,\to}$. Since the model $M$ is rooted, there exists $w \in U(n)_{\wedge,\to}$ such that $M = {\uparrow} w$.
We define $\beta(F) := s(\psi_w) = s(\varphi_w) \to s(\theta_w)$ and prove that $\beta(F)$ satisfies the required property.

First note that, as follows from the proof of Theorem~\ref{thm:definable}, $s(\varphi_w)$ defines the up-set of $U(n)_{\wedge,\to}$ generated by $w$ 
and $s(\theta_w)$ defines
the up-set of $U(n)_{\wedge,\to}$ generated by the set of proper successors of $w$. 
Therefore, $w$ is the only point of $U(n)_{\wedge,\to}$ that satisfies $s(\varphi_w)$ and refutes $s(\theta_w)$.

Let $v\in N$ be such that $N,v\not\models \beta(F)$. %By definition, we can pick an (admissible) colouring $c'$ on $G$ and $v \in G$ such that $N,v\not\models  \beta(F)$, where we write $N$ for the model $(G,c')$. 
Since $N$ is descriptive, 
we can find a successor $u$ of $v$ such that $N ,u\models  s(\varphi_w)$, $N ,u\not\models   s(\theta_w)$ and every proper successor of $u$ satisfies $s(\theta_w)$ (see, e.g., \cite[Thm. 2.3.24]{NBezh}).
By Lemma~\ref{lem: 8}, this implies that $u\in N^s$. 
Let $f:N^s \to U(n)_{\wedge,\to}$ be the unique p-morphism as in Proposition~\ref{prop:findepmap}.
Because $u\in N^s$, $s(\varphi_w)$ is a 
$(\wedge,\to)$-formula and  $N,u\models  s(\varphi_w)$, Proposition~\ref{cor: 17} entails that
 $U(n)_{\wedge,\to}, f(u)\models s(\varphi_w)$.
By the same argument we also have that 
$U(n)_{\wedge,\to}, f(u)\not\models  s(\theta_w)$.
Thus, we obtain that $U(n)_{\wedge,\to}, f(u)\models  s(\varphi_w)$ and  $U(n)_{\wedge,\to}, f(u)\not\models  s(\theta_w)$. 
We have shown in the previous paragraph that this implies 
$f(u) = w$. Therefore, as $f$ is a p-morphism, we obtain that $F$ is a p-morphic image of $N^s$.

For the other direction, let $f : N^s \to M$ be a surjective p-morphism. Since $f$ is surjective, pick $u \in N^s$ such that $f(u) = w$. As $w$ satisfies $s(\varphi_w)$ and refutes $s(\theta_w)$, and both are $(\wedge,\to)$-formulas, the same argument as above gives that $N, u \models s(\varphi_w)$ and $N, u \not\models s(\theta_w)$. Hence, $N \not\models \beta(F)$.
 \qed
\end{proof}

The formula $\beta(F)$ defined in Theorem~\ref{thm: 5} is 
called the \emph{subframe formula of $F$}.

Recall that, for any formula $\varphi$ in $n$ variables, we say a Heyting algebra $A$ {\it validates the equation $\varphi \approx 1$}, notation $A \models \varphi \approx 1$, if $\overline{v}(\varphi) = 1$ under each assignment $v : \{p_1,\dots,p_n\} \to A$. If there is an assignment $v$ under which $\overline{v}(\varphi) \neq 1$, we say that $A$ {\it refutes the equation $\varphi \approx 1$}.
\begin{corollary}
Let $F$ be a finite rooted frame and $A$ its Heyting algebra of up-sets. 
Then for each Heyting algebra $B$  we have 
$$
B \not\models \beta(F) \approx 1 \ \ \iff \ \ \mbox{there is a $(\wedge,\to)$-embedding $A \hookrightarrow B$}
$$
\end{corollary}
\begin{proof}
It follows from the proof of Theorem~\ref{thm: 5} that the model $(F,c)$ refutes 
$\beta(F)$. This means that, in the Heyting algebra $A = \mathcal{U}(F)$, the formula $\beta(F)$ does not evaluate to $1$ under the assignment $v : p_i \mapsto c^*(p_i)$.
Suppose that there is a $(\wedge,\to)$-embedding $i : A \hookrightarrow B$. Since $\beta(F)$ is a $(\wedge, \to)$-formula, under the assignment $i \circ v$, the formula $\beta(F)$ does not evaluate to $1$ in $B$.
Conversely, suppose that $B \not\models \beta(F) \approx 1$, under an assignment $v$. Let $G$ be the descriptive frame with
$B$ as its algebra of admissible up-sets. The assignment $v$ yields an admissible colouring $c'$ on $G$ with the property that $N = (G,c')\not\models \beta(F)$. By Theorem~\ref{thm: 5}, this implies in particular that $F$  is a p-morphic image of $N^s$.
It now follows from Theorem~\ref{thm:gensubdual} that $A$ is $(\wedge, \to)$-embedded into $B$.
\qed
\end{proof}

\begin{remark}
Subframe formulas axiomatize a large class of logics having the finite model property \cite[Ch.\ 11]{ChaZak1997}. The frames of these logics are closed under taking 
subframes. Alternatively varieties of Heyting algebras corresponding to these logics 
are closed under $(\wedge, \to)$-subalgebras. 
There exist many different ways to define 
subframe formulas for intuionistic logic: model-theoretic \cite[Ch.\ 11]{ChaZak1997}, algebraic \cite{BG07}, \cite{BezhBezhCan}, and 
via the so-called NNIL formulas \cite{VdJvB}. Theorem~\ref{thm: 5} gives a new way to define subframe formulas. 
The proof of this theorem shows that the same way de Jongh formulas for intuitionistic logic correspond to Jankov formulas
\cite{NBezh}, de Jongh formulas for the $(\wedge,\to)$-fragment of intuitionistic logic correspond to subframe formulas. 
This provides a different perspective on the interaction of de Jongh-type formulas and frame-based formulas such as Jankov formulas, subframe 
formulas etc.  
\end{remark}

We finish this section by applying the results of this paper 
to show that the uniform IPC-interpolant, as defined by Pitts \cite{Pit1992}, of a meet-implication formula 
is not necessarily equivalent to a meet-implication formula. 
\begin{example}\label{exa:unifint}
As can be readily checked, the uniform interpolant of the formula $p\to (q \to p)$ in IPC with respect to the variable $p$ 
is the formula $\neg\neg p$. We will use the characterization in Theorem~\ref{thm:characterization} to prove that $\neg \neg p$ is not equivalent to a $(\wedge,\to)$-formula. Namely, if there were a $(\wedge,\to)$-formula $\varphi$ equivalent to $\neg\neg p$, then in particular the up-set $U$ defined by the formula $\neg\neg p$ in the 1-universal model of IPC (see figure 1 below) would be $(\wedge, \to)$-definable. It thus suffices to show that $U$ is not $(\wedge,\to)$-definable. 
To see this, note that $U(1)^s = \max v(p)^c = \{x_1,x_2\}$, and these two points are bisimilar in $U(1)^s$.
Since $x_2 \in U$ but $x_1 \not\in U$, $U$ does not satisfy (3b) in Theorem~\ref{thm:characterization}, and is therefore not $(\wedge,\to)$-definable.

We now also prove that the least $(\wedge,\to)$-definable up-set of $U(1)$ containing $U$ is $U(1)$ itself. 
Indeed, let $W$ be a $(\wedge,\to)$-definable up-set which contains $U$.
Then, by the above, $x_1$ 
belongs to $W$. It then easily follows from (3a) in 
Theorem~\ref{thm:characterization} that
every colour $0$ point of $U(1)$ must also belong to $W$.
Thus, $W = U(1)$. This argument shows, via semantics, that the 
$(\wedge,\to)$-formula which is a uniform interpolant of $p\to (q\to p)$ is 
$\top$. 
We refer to \cite{dJZh14} for more details on uniform interpolation in fragments of intuitionistic logic. 
\qed

\begin{center}
\begin{tikzpicture}
\po{0,0}
\draw node[above,yshift=2pt] at (0,0) {$1$};
\po{2,0}
\draw node[above,yshift=2pt] at (2,0) {$0$};
\draw node[right] at (2,0) {$x_1$};
\po{0,-1}
\draw node[left] at (0,-1) {$0$};
\draw node[right,yshift=5pt] at (0,-1) {$x_2$};
\po{2,-1}
\draw node[right] at (2,-1) {$0$};
\po{0,-2}
\draw node[left] at (0,-2) {$0$};
\po{2,-2}
\draw node[right] at (2,-2) {$0$};
\po{0,-3}
\draw node[left] at (0,-3) {$0$};
\po{2,-3}
\draw node[right] at (2,-3) {$0$};
\po{0,-4}
\draw node[left] at (0,-4) {$0$};
\po{2,-4}
\draw node[right] at (2,-4) {$0$};

\draw[thick,dashed] (0,-.5) ellipse (.6cm and 1cm);
\draw node at (-.48,.62) {$U$};

\li{(0,0)--(0,-4.3)}	%line down left
\li{(2,0)--(2,-4.3)}	%line down right

\li{(0,0)--(2,-1)}		%lines left-to-right
\li{(0,-1)--(2,-2)}
\li{(0,-2)--(2,-3)}
\li{(0,-3)--(2,-4)}
\li{(0,-4)--(2,-5)}

\li{(2,0)--(0,-2)}
\li{(2,-1)--(0,-3)}
\li{(2,-2)--(0,-4)}
\li{(2,-3)--(0,-5)}
\draw node[rotate=90] at (1,-5.2) {$\dots$};

\end{tikzpicture}

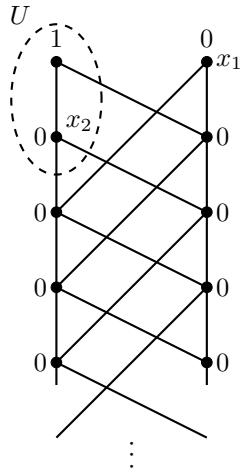
Figure 1: The 1-universal model, $U(1)$, also known as the Rieger-Nishimura

ladder, with $U(1)^s = \{x_1,x_2\}$ and $U = v(\neg\neg p)$.
\end{center}
\end{example}

\section{Conclusions and future work}\label{sec:conclusion}

In this paper we studied the $(\wedge,\to)$-fragment of intuitionistic logic via methods of duality theory. We gave an alternative proof
of Diego's theorem and characterized $(\wedge,\to)$-definable up-sets of the $n$-universal model of intuitionistic 
logic, using duality as our main tool. Interestingly, we were able to directly use finite duality for distributive lattices and adjunction properties such as the Frobenius property (Lemma~\ref{lem:frob}), without resorting to any of the existing dualities for implicative meet-semilattices. We expect that the techniques developed in Section~\ref{sec:sepmeetimp} could be extended to the infinite setting in order to give a unified account of the different dualities that exist in the literature for implicative meet-semilattices, e.g., \cite{Kohler}, \cite{GBezhJans} and \cite{BezhBezhCan}. We leave this as an interesting question for future work.

The characterization of $(\wedge,\to)$-definable up-sets that we gave in Theorem~\ref{thm:characterization} can be considered as a first step towards solving the complicated problem of characterizing all IPC-definable up-sets of $n$-universal models. 
This problem is linked to the following interesting question. 
In \cite{Ghi1992} free Heyting algebras are described from free distributive lattices via step-by-step approximations of 
the operation $\to$. 
In \cite{NBezhGeh}, the authors explained how the construction in \cite{Ghi1992} can be understood via (finite) duality for distributive lattices. This begs the question whether one can use duality for implicative meet-semilattices to build free Heyting algebras, starting from free implicative meet-semilattices and approximating the operation of disjunction, $\vee$, step-by-step.
The results of this paper can be considered as the first (or actually zeroth) step 
of such a step-by-step construction.

Finally,  we note that \cite{Lex} and \cite{TZh13} study $n$-universal models in other fragments 
of intuitionistic logic. We leave it to future work to investigate how the duality methods of this paper relate to 
the methods developed in  \cite{Lex} and \cite{TZh13}.

\section*{Acknowledgements}
We are thankful to Mai Gehrke for many inspiring discussions on this paper. We also thank the referees for many useful suggestions.

\bibliographystyle{splncs}
\bibliography{universalmodel}

\end{document}